\documentclass[11pt, oneside]{article}

\usepackage{graphicx}
\usepackage[cmyk]{xcolor}
\usepackage{amssymb,amsmath,amsthm}

\newtheorem{theorem}{Theorem}

\newtheorem{lemma}{Lemma}
\newtheorem{cor}{Corollary}
\newtheorem{proposition}{Proposition}
\newtheorem{remark}{Remark}

\usepackage{tikz}
\usetikzlibrary{decorations.markings}
\usepackage{ucs} 
\usepackage[utf8x]{inputenc}
\usepackage{pgflibraryarrows}
\usepackage{pgflibrarysnakes}
\usepackage{pstricks,pst-node,pst-tree,pstricks-add}
\usepackage{pgfplots}
\usepackage{tkz-tab}
\usepackage{caption}

\usepackage[colorlinks=true,allcolors=blue]{hyperref}
\title{$N$-body choreographies on a $p$-lima\c{c}on curve}

\begin{document}
\maketitle
	\markboth{M. Fernandez-Guasti, T.  Fujiwara, E.  P\'{e}rez-Chavela, and  S. Q.  Zhu}{Choreographies on $p$-lima\c{c}ons}
\author{\begin{center}
		{ Manuel Fernandez-Guasti$^1$, Toshiaki Fujiwara$^2$, Ernesto P\'{e}rez-Chavela$^3$, Shuqiang Zhu$^{4}$}\\	
		\bigskip
		$^1$ Lab. de Óptica Cuántica, Depto. de Física, Universidad \\
		Autónoma Metropolitana, mfg@xanum.uam.mx\\
		$^2$ College of Liberal Arts and Sciences, Kitasato University, 1-15-1 Kitasato, Sagamihara, Kanagawa 252-0329, Japan,
		fujiwara@kitasato-u.ac.jp\\
		$^3$Department of Mathematics, Instituto Tecnol\'ogico Aut\'onomo\\ de México (ITAM), ernesto.perez@itam.mx\\
		$^4$School of  Mathematics,  Southwestern University of Finance and Economics, Chengdu, China, zhusq@swufe.edu.cn\\
	\end{center}

\begin{abstract}
	We consider  an $N$--body problem under a harmonic potential of the form  $\frac{1}{2}\sum \kappa_{jl} |q_j-q_l|^2$. 
A $p$-lima\c{c}on curve 
is a planar curve parametrized by $t$ given by $a(\cos t,\sin t)+b(\cos pt, \sin pt)$, where
$a,b\in \mathbb{R}$, $p \in \mathbb{Z}$,
and $t \in [0,2\pi]$.  We study $N$-body choreographic motions constrained to a $p$-limaçon curve and establish necessary and sufficient conditions for their existence. Specifically, we prove that choreographic motions exist if and only if $p/N, (p \pm 1)/N \notin \mathbb{Z}$. 
Under an additional symmetry assumption on the force coefficients, we further refine these conditions. 
We also analyze the occurrence of collisions, showing that for given $p$ and $N$, at most $2(N-1)$ choices of $a/b$ lead to collisions. Furthermore, we find additional conserved quantities. 
\end{abstract}



\section{Introduction}
A choreography is a periodic solution of the $N$--body problem, where the particles move on the same curve with the same time 
spacing between any two consecutive particles. The first choreography was found by J.L. Lagrange in 1772 for the Newtonian three-body problem with equal masses \cite{Lagrange}, on it, the masses are located at the vertices of an equilateral triangle. In 1993, also in the Newtonian three-body 
problem, C. Moore found numerically the well-known figure eight choreography \cite{Moore}. Some years later, in 2000, Chenciner and Montgomery, by using variational methods, gave an analytical proof of the existence of this choreography \cite{Chenciner}. In 2002, Chenciner, Gerver, Montgomery and Sim\'o, proved that for the strong force 
potential, for any $N \geq 3$, there is a choreography solution for the equal mass case, different from the circular one \cite{Chenciner2}. Since then, many works have been published showing beautiful  new choreographies for the $N$--equal masses problem,  cf. \cite{BarutelloTerracini, BarutelloFerrarioTerracini, Chen, FerrarioTerracini, Ferrario, FerrarioPortaluri, Shibayama, TerraciniVenturelli, WangZhang, Yu2016} and references therein.  

Although the existence of such choreographies is shown, 
very few  analytical choreographic solutions are presently known. In  the Newtonian case, the only analytical choreographies are  $N$ equal masses rotating in a regular N-gon in a plane within a circumscribed circle.

In 2003, Fujiwara T.,  Fukuda H., Ozaki H.,   studied analytical choreographies on the lemniscate curve for the $3$--equal masses problem, when the particles move under the action of an inhomogeneous potential \cite{Fujiwara0}. 
Later, in 2019,
Vieyra \cite{Vieyra}, and Turbiner \& Vieyra \cite{TurbinerVieyra}
extended the result to $N$-body choreographies.
The main idea to extend the three-body problem to  the $N$-body problem
is to keep  the interactions between particles with consecutive indices 
$V_{j,j\pm 1}$ and set the other potentials vanish.
Then, one particle interacts only  with
 the consecutive indexed 
two particles.
Namely, the equation of motion for this particle
is the same as for the three-body problem.
Similarly, keep 
 the interactions between the particles with next-consecutive index 
$V_{j,j\pm 2}$ and set the other potentials vanish and so on.

The lesson to be learned from this is that, if the symmetry imposed on the system is relaxed from 
$S_N$ (the permutation group of $N$ particles,
that produce the ``universal gravity'')
to $Z_N$ (cyclic group generated by cyclic rotation of $j \to j+1$),
the system's possibilities increase and the equations of motion can have a greater variety of solutions.

Recently, Fernandez-Guasti \cite{Fernandez}
found another  analytical choreography,  the four-body choreography
on a special lima\c{c}on of Pascal,
\begin{equation*}
q(t)=(x,y)=2^{-1}(\cos t, \sin t) + 2^{-1}(\cos(2 t),\sin(2 t)).
\end{equation*}
In the same paper, he  also found several constants of motion other than the well-known ones for this motion  \cite{Fernandez}. 
He imposed a $Z_4$ symmetry to the system.
Then, the interaction terms are grouped by
 the consecutive
$V_{j,j\pm 1}$,
and 
 the next consecutive 
$V_{j, j\pm 2}$.
More precisely, he assumed the potential energy $V$ as
\[
V=\sum_{j<l} \frac{\kappa_{jl}}{2}|q_j-q_l|^2,\quad \kappa_{jl}=\kappa_{lj},
\]
with
$\kappa_{01}=\kappa_{12}=\kappa_{23}=\kappa_{30}=1$ (attractive),
and $\kappa_{02}=\kappa_{13}=-1/2$ (repulsive).

Inspired by his work,
we  study $N$-body equal mass choreographies on a 
$p$-lima\c{c}on curve in $\mathbb{R}^2$ given by 
\begin{equation}\label{eq:limacon}
q(t)=(x,y)=a(\cos t, \sin t) + b(\cos(pt),\sin(pt)),
\end{equation}
with $a,b \in \mathbb{R}$ and $p \in \mathbb{Z}\setminus \{1,0,-1\}$.
The curve for 
$p=2$ is is called lima\c{c}on, and for $p=-2$ is called trefoil. In general the curves for $p>0$, are called \emph{$p$-lima\c{c}on } and for $p<0$ are called  \emph{$(|p|+1)$-folium} (see \cite{Lawrence, Dana}). 
Since our analysis is valid for all $p \in \mathbb{Z} \setminus \{1,0,-1\}$, we decided to call both curves as \emph{$p$-lima\c{c}on}.

We assume that the masses are moving under the influence of the harmonic potential 
\begin{equation}\label{eq:potential}
V=\sum_{j<l} \frac{\kappa_{jl}}{2}|q_j-q_l|^2, \quad \kappa_{jl}=\kappa_{lj},
\end{equation}
where the force coefficients $\kappa_{jl}$  are  to be determined.

If we impose the $S_N$ invariance,
the permutation group of $N$-bodies,
all  the force coefficients $\kappa_{jl}$ must be equal.
But, 
by the fact that 
\[
\sum_{j<k} |q_j-q_k|^2
=N\sum_{k} |q_k|^2\  \mbox{ if }\ \sum_k q_k=0,
\]
we observe that the above potential expresses the 
assembly of  $N$ free  harmonic oscillators.
Then, the solution orbits under this potential are 
a
set of $N$ ovals.

On the other hand, if we relax the symmetry to $Z_N$,
the symmetric rotation group generated by $j\to j+1$, 
the force coefficients $\kappa_{jl}$
can be separated  to some independent groups,
$\kappa_1$ for
 the consecutive indexed interaction, 
i.e., $\kappa_1=\kappa_{0,1}=\dots, =\kappa_{N-1,0}$, 
$\kappa_2$ for
 the next-consecutive indexed interaction, 
i.e., $\kappa_2=\kappa_{0,2}=\dots, =\kappa_{N-2,0}=\kappa_{N-1,1}$, and so on.  
Then, the equation of motion will have rich solutions, 
as shown by Fernandez-Guasti \cite{Fernandez} in the case of $N=4$.
However, 
if $N=3$, $Z_N$ invariance demands that all $\kappa_{jl}$ must be equal.
That means free harmonic oscillators.
So, we consider $N\ge 4$. 

\emph{The main problem of this work is to determine, under the above setting,   the values of
	\[
	p, \ N, \ a, \ b, \ \kappa_1, \dots, \ \kappa_{\lfloor \frac{N}{2}\rfloor},
	\]  
	that ensure an $N$--body choreography on a \( p \)-limaçon curve.}

Assume that $p\ne 0$ and $p\ne \pm1$.    We establish a fundamental condition on \( p \) and \( N \) (Theorem \ref{thm1}), showing that choreographies exist if and only if  
\[
p/N, (p \pm 1)/N \notin \mathbb{Z}.   \ 
\]  
When this condition is satisfied, we can determine the force coefficients \( \kappa_1, \dots, \kappa_{\lfloor \frac{N}{2}\rfloor} \). This condition is further refined under the following additional symmetry assumption 
(Theorem \ref{thm2})
\[ \kappa_1 = \kappa_3 = \kappa_5 = \dots, \quad   \kappa_2 = \kappa_4 = \kappa_6 = \dots.   \]

In some of the choreographic motions studied, collisions between the bodies may occur. We  show that for a given \( p \) and \( N \), at most \( 2(N-1) \) values of \( a/b \) lead to collisions. Additionally, we identify further conserved quantities beyond the standard center of mass, angular momentum, and energy.

After the introduction the paper is organized as follows: In 
Section \ref{Sec:Lagrangian} we give the Lagrangian and the respective equations of motion. Section \ref{Sec: Nece} examines the center of mass, angular momentum, moment of inertia, and kinetic energy, leading to necessary conditions on \( p \) and \( N \). Section \ref{Sec:Main-1} presents our first main result, Theorem \ref{thm1}. Section \ref{Sec:Main-2} establishes our second main result, Theorem \ref{thm2}.  Collisions are analyzed in Section \ref{Sec:Collision}. Finally, in Section \ref{Sec:Constants}, we discuss additional conserved quantities.

\section{Basic settings}\label{Sec:Lagrangian}
Consider the $p$-lima\c{c}on curve in $\mathbb{R}^2$ parametrized by $t$
given by \eqref{eq:limacon}
with $a,b \in \mathbb{R}$ and 
$p \in \mathbb{Z}\setminus \{0, \pm 1\}$ (See Figure \ref{fig:limacon}). 
We assume $ab\ne 0$, since if either 
 $a$ or $b$ vanish,  the motion reduces to a circular trajectory, 
which is not interesting.
We also exclude $p=0, \pm 1$,
as these cases correspond to a circle or an oval rather than a lima\c{c}on curve.

\begin{figure}[h]
		\hspace{0.2cm} 
\begin{tikzpicture}
\begin{axis}[width=2.6cm, height=2.6cm,
scale only axis,
xmin=-2.3, xmax=2.3,
ymin=-2.3, ymax=2.3,
xlabel={$x$}, ylabel={$y$},
samples=200,
domain=0:2*pi,
xtick=\empty, ytick=\empty,
enlargelimits=false,
axis lines=middle
]
\addplot[purple, thick, samples=300, domain=0:2*pi, smooth] (
{cos(deg(x)) + 0.8*cos(2*deg(x))},
{sin(deg(x)) + 0.8*sin(2*deg(x))});
\end{axis}
\end{tikzpicture}
\hspace{0.1cm} 
\begin{tikzpicture}
\begin{axis}[width=2.6cm, height=2.6cm,
scale only axis,
xmin=-2.3, xmax=2.3,
ymin=-2.3, ymax=2.3,
xlabel={$x$}, ylabel={$y$},
samples=200,
domain=0:2*pi,
xtick=\empty, ytick=\empty,
enlargelimits=false,
axis lines=middle
]
\addplot[purple, thick, samples=300, domain=0:2*pi, smooth] (
{cos(deg(x)) + 0.8* cos(-2*deg(x))},
{sin(deg(x)) + 0.8* sin(-2*deg(x))}
);
\end{axis}
\end{tikzpicture}
\hspace{0.1cm} 
\begin{tikzpicture}
\begin{axis}[width=2.6cm, height=2.6cm,
scale only axis,
xmin=-2.3, xmax=2.3,
ymin=-2.3, ymax=2.3,
xlabel={$x$}, ylabel={$y$},
samples=200,
domain=0:2*pi,
xtick=\empty, ytick=\empty,
enlargelimits=false,
axis lines=middle
]
\addplot[purple, thick, samples=300, domain=0:2*pi, smooth] (
{cos(deg(x)) +0.8* cos(3*deg(x))},
{sin(deg(x)) + 0.8*sin(3*deg(x))}
);
\end{axis}
\end{tikzpicture}
\hspace{0.1cm} 
\begin{tikzpicture}
\begin{axis}[width=2.6cm, height=2.6cm,
scale only axis,
xmin=-2.3, xmax=2.3,
ymin=-2.3, ymax=2.3,
xlabel={$x$}, ylabel={$y$},
samples=200,
domain=0:2*pi,
xtick=\empty, ytick=\empty,
enlargelimits=false,
axis lines=middle
]
\addplot[purple, thick, samples=300, domain=0:2*pi, smooth] (
{cos(deg(x)) + 0.8*cos(-3*deg(x))},
{sin(deg(x)) + 0.8*sin(-3*deg(x))}
);
\end{axis}
\end{tikzpicture}
\caption{Four p-lima\c{c}on curves for $a=1.2, b=1$. From left to right, the values of p is $p=2$,  $p=-2$, $p=3$,  and $p=-3$.  
 }
	\label{fig:limacon}
\end{figure}
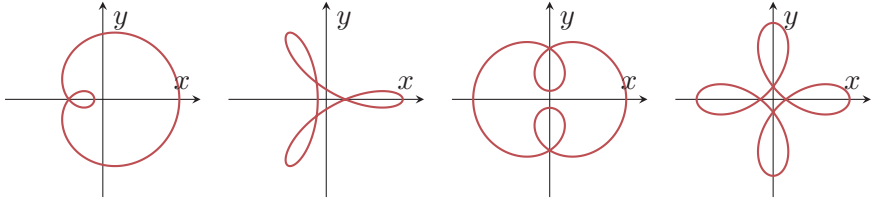

We will use  the following complex notation,
\begin{equation}\notag
z=x+iy=a e^{it} + b e^{ipt}, \quad 
  \quad 
z^\dag=x-iy=a e^{-it} + b e^{-ipt}.
\end{equation}

Now, we consider the $N$-body  choreographic motion on the $p$-lima\c{c}on curve for $N\ge 4$. 
\begin{equation}\label{choreograph1}
q_k(t)=q(t+ k \frac{2\pi}{N}), \quad k=0,1,2,\dots, N-1.
\end{equation}
Namely,
\begin{equation}\notag
q_k(t)=a\left(\cos\left(t+\frac{2k\pi}{N}\right), \sin\left(t+\frac{2k\pi}{N}\right)\right)
	+ b\left(\cos\left(pt+\frac{2kp\pi}{N}\right), \sin\left(pt+\frac{2kp\pi}{N}\right)\right).
\end{equation}

Or, using the complex variables, 
\begin{equation}\notag
\begin{split}
z_k(t)&=ae^{it+i2k\pi/N}+be^{ipt+i2kp\pi/N}=ae^{it}\omega^k+be^{ipt}\omega^{pk},\\
\end{split}
\end{equation}
where
\begin{equation}\notag
\omega=e^{i2\pi/N}\ne 1 \quad \mbox{ and } \quad \omega^N=1.
\end{equation}

Consider the following Lagrangian for $N$ masses with $m_k = 1$ for all $k$, 
\begin{equation}\label{Lagrangian}
\begin{split}
L&=K-V,\\
K&=\frac{1}{2}\sum_k |\dot{q}_k|^2,\quad
V=\frac{1}{2}\sum_{j<l} \kappa_{jl}|q_j-q_l|^2,\quad \kappa_{jl}=\kappa_{lj},
\end{split}
\end{equation}
where $K$ is the kinetic energy, $V$ is the potential energy, and $\kappa_{jl}$ are the force coefficients.

As stated in the introduction,
we impose the $Z_N$ symmetry to this potential.
That is,  the consecutive indexed interaction terms have the same $\kappa_1$,
the next consecutive indexed terms have another equal $\kappa_2$, and so on.
Namely,
$\kappa_1=\kappa_{0,1}=\dots=\kappa_{j,j+1}=\dots=\kappa_{N-1,0}$;
$\kappa_2=\kappa_{0,2}=\dots=\kappa_{j,j+2}=\dots=\kappa_{N-2,0}=\kappa_{N-1,1}$,
and
\[
\kappa_{\lfloor N/2\rfloor}=
\begin{cases}
\kappa_{0,N/2}=\kappa_{1,N/2+1}=\dots=\kappa_{N/2-1,N-1}&\mbox{ for even }N\\
\kappa_{0,(N-1)/2}=\kappa_{1,(N+1)/2}=\dots=\kappa_{N-1,(N-3)/2}&\mbox{ for odd }N.
\end{cases}
\]
See Figure \ref{figKappas}.
 For $\kappa_{\lfloor N/2\rfloor}$,
there are $N/2$ bonds for even $N$,
and $N$ bonds for odd $N$.
\begin{figure}[h!] 
   \centering
   \includegraphics[width=4cm]{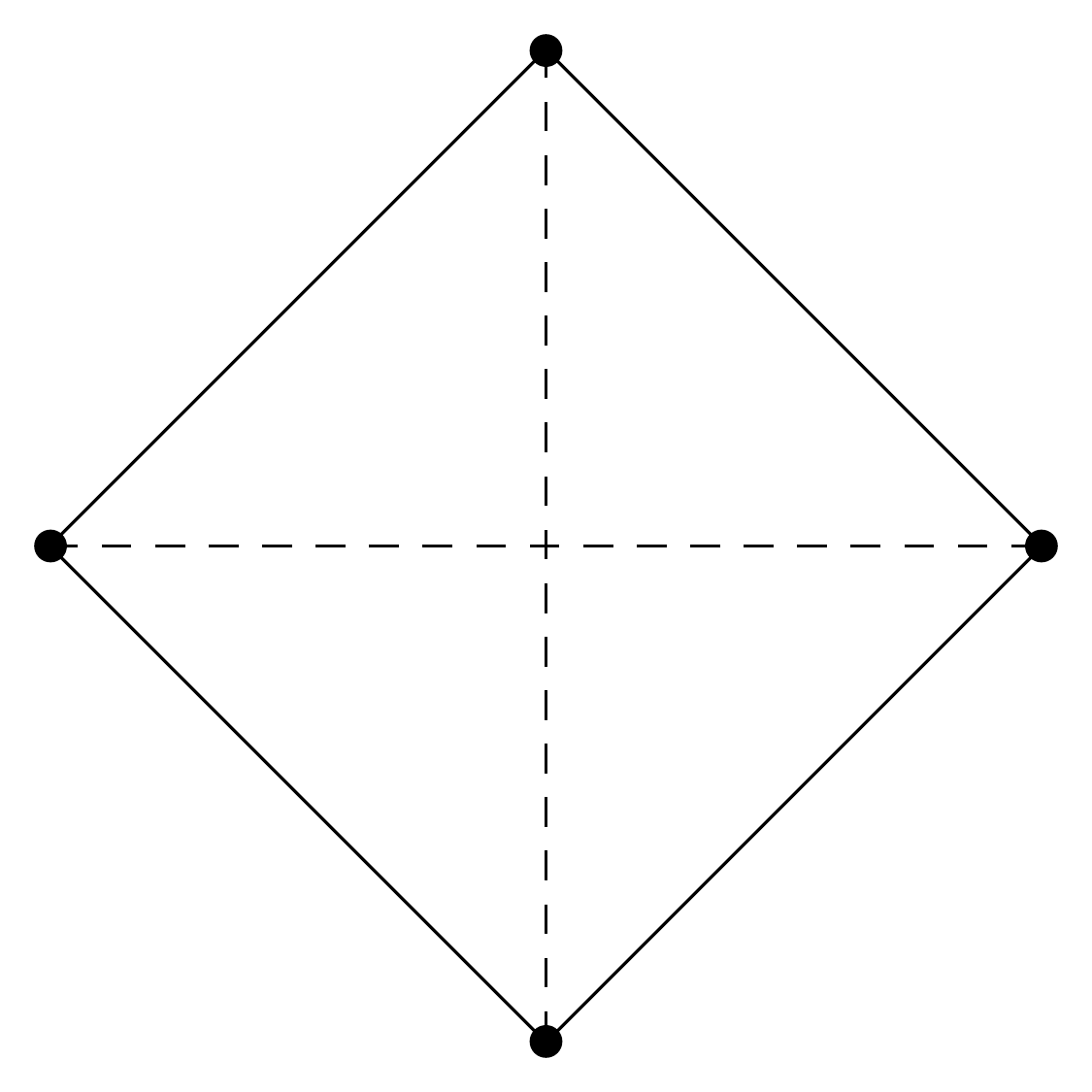}
   \includegraphics[width=4cm]{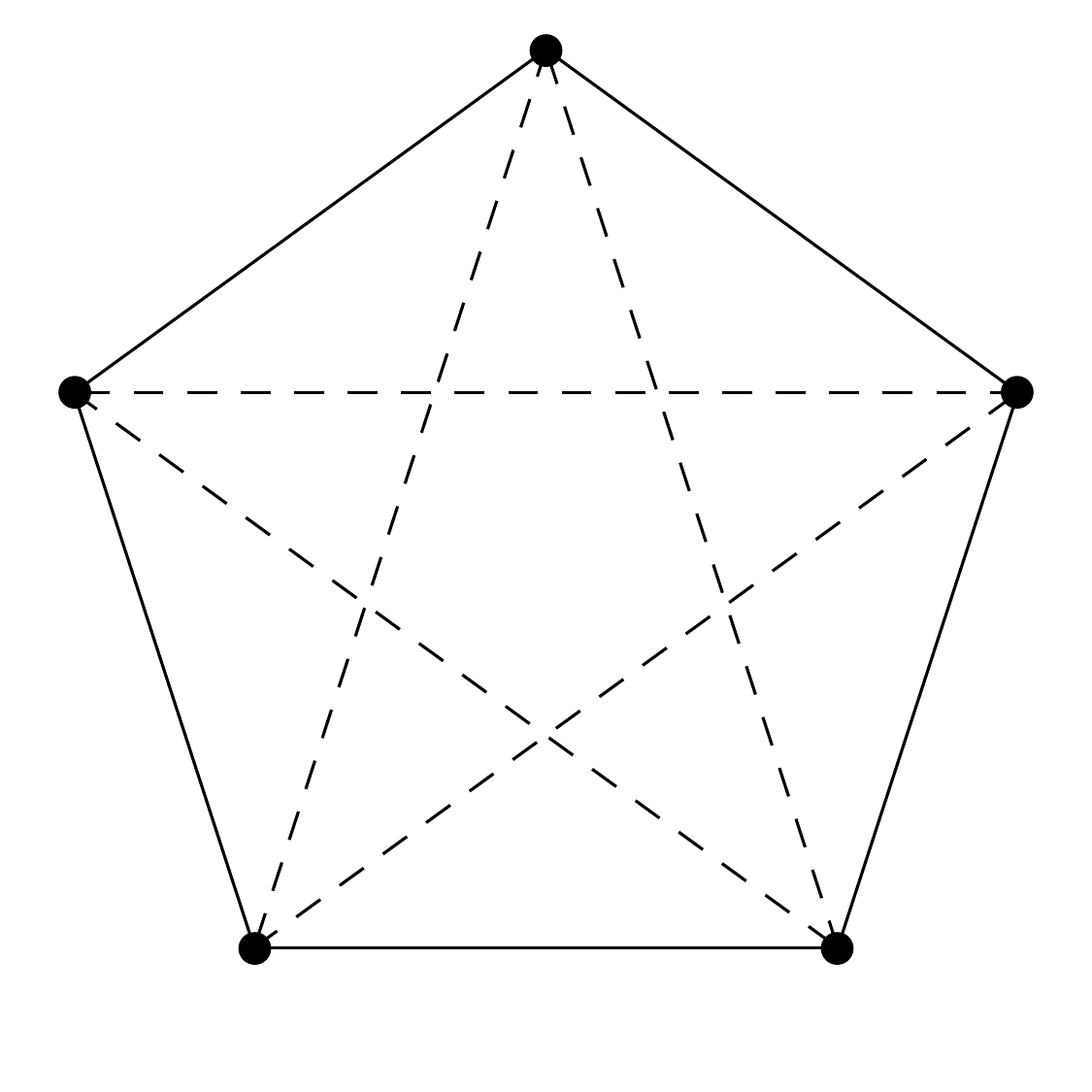}
   \includegraphics[width=4cm]{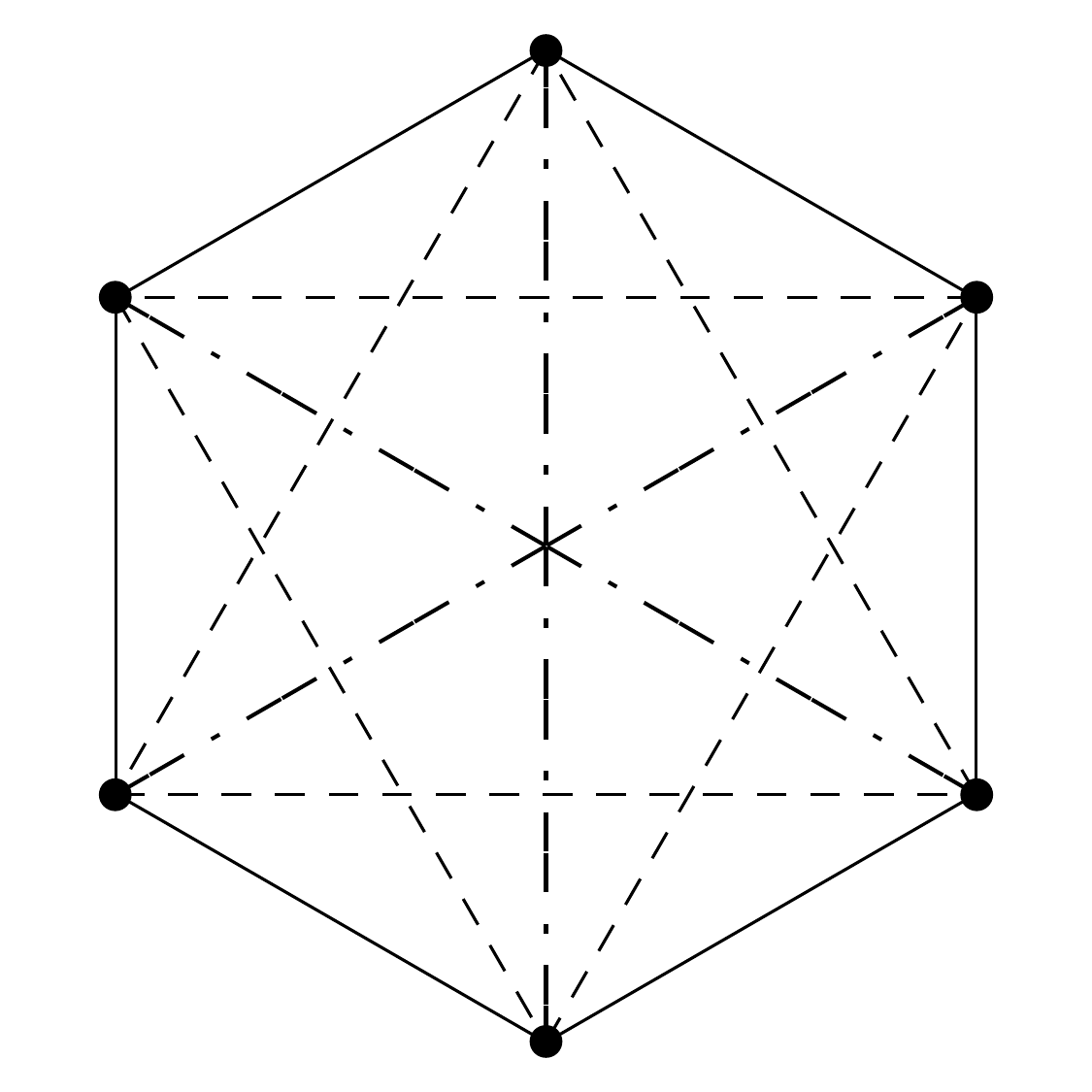}\\
   \includegraphics[width=4cm]{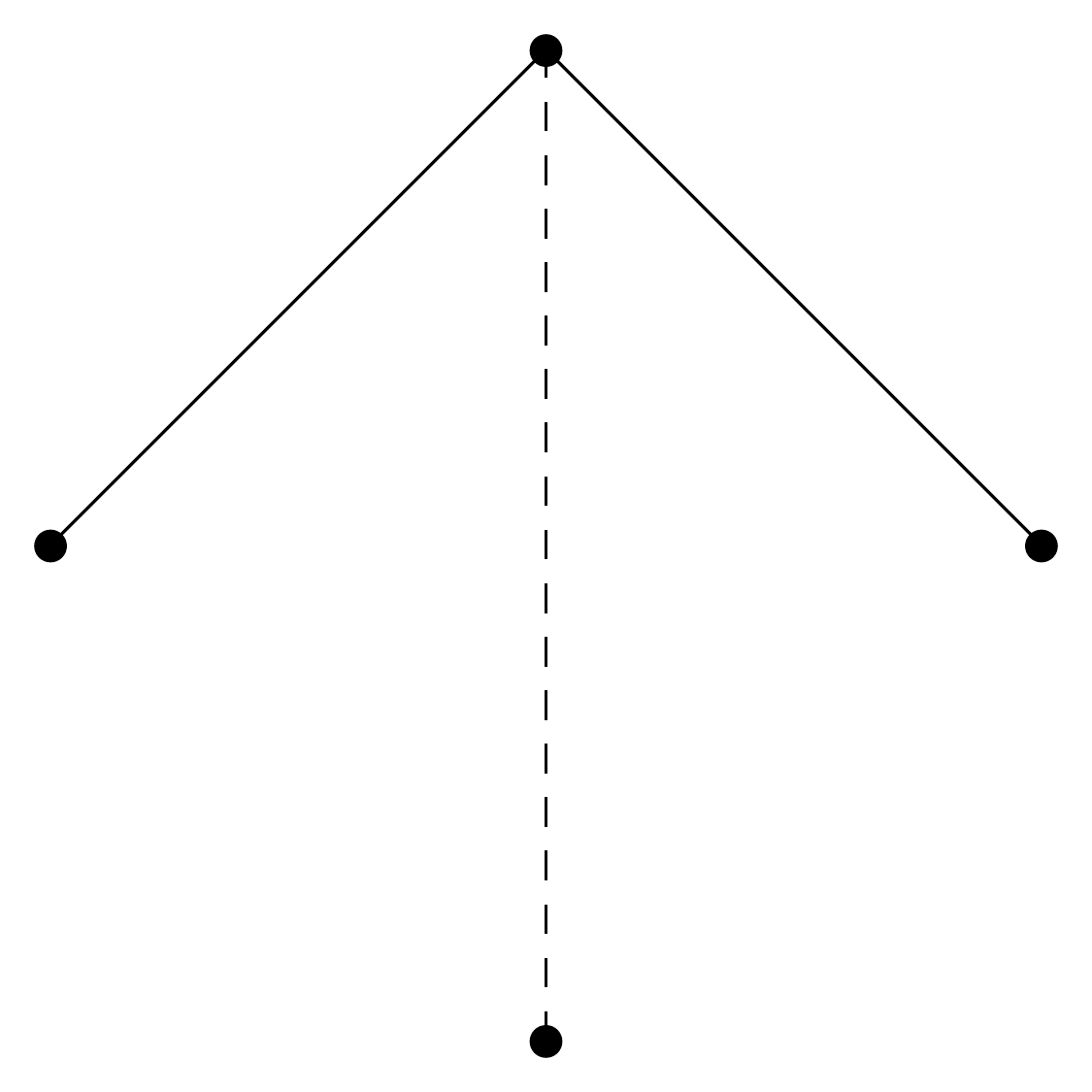}
   \includegraphics[width=4cm]{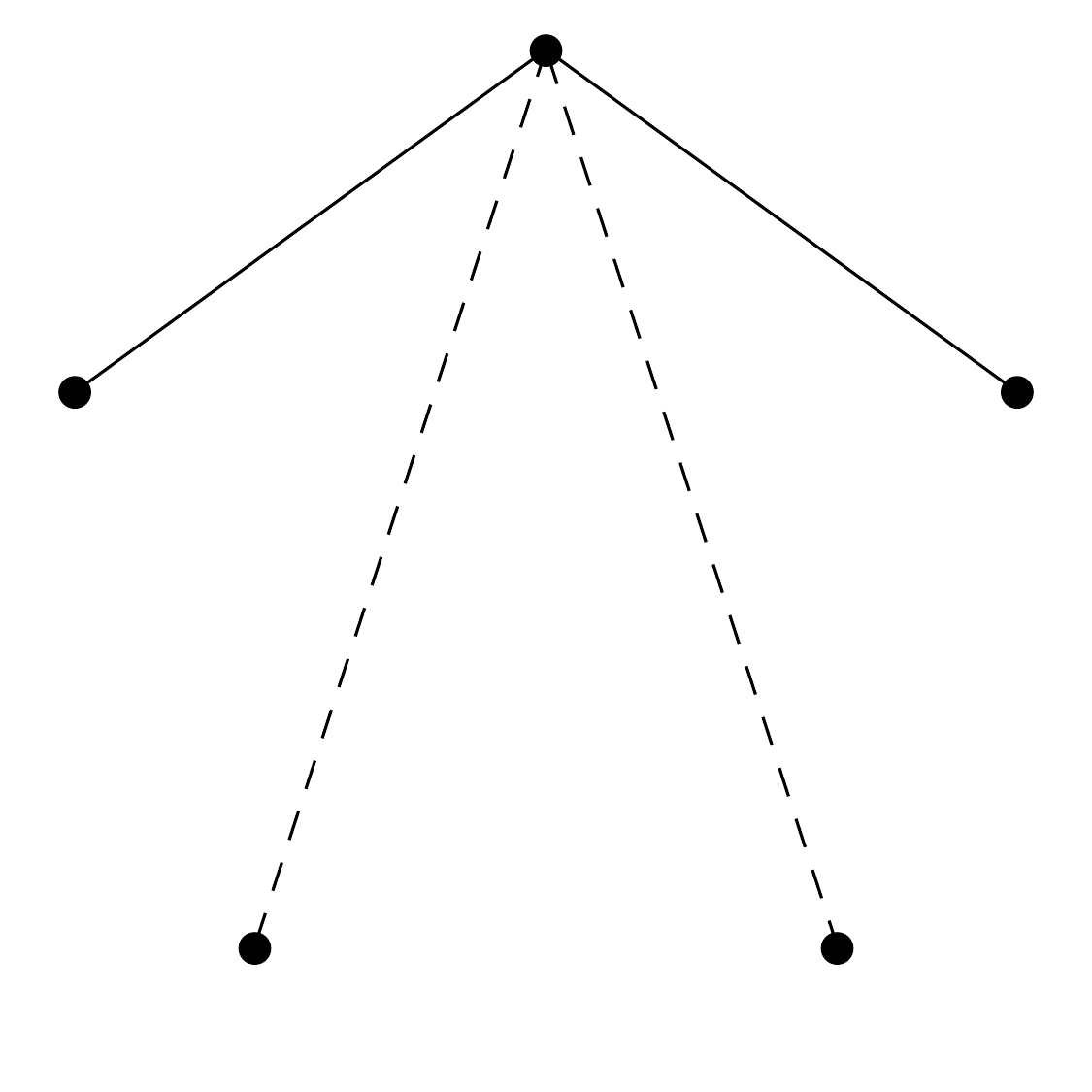}
   \includegraphics[width=4cm]{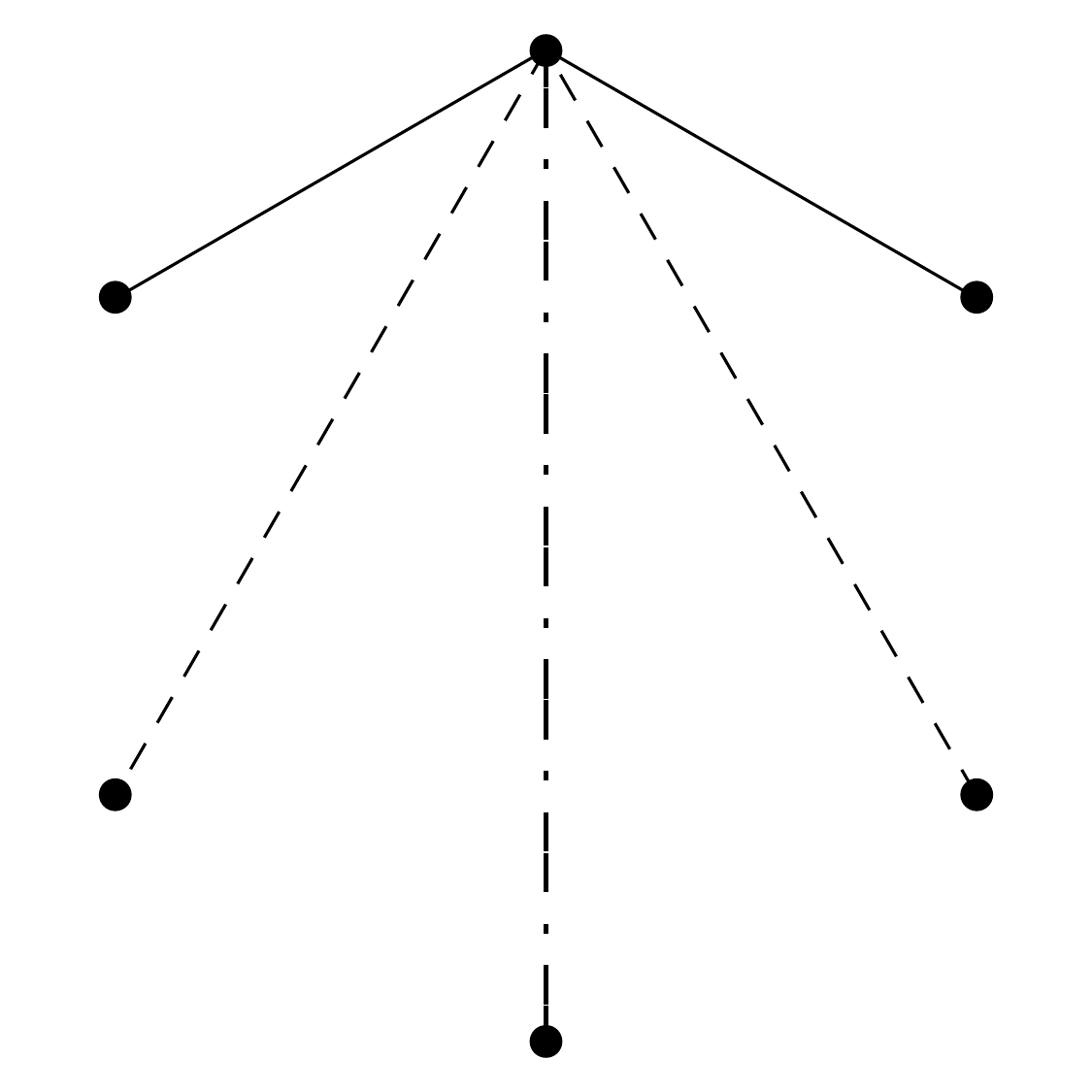}
   \caption{
   Upper: Interactions between $q_j$ and $q_k$.
   Lower: Interactions between $q_0$ (the top point) and $q_j$.
   From left to right, $N=4,5,6$.
   The solid, dashed, and dot-dashed  lines represents $\kappa_1, \kappa_2$, and $\kappa_3$ respectively.
   }
   \label{figKappas}
\end{figure}

The equations of motion are
\begin{equation}\label{eqOfMotion1}
\ddot{q}_j 
= -\frac{\partial V}{\partial q_j}
=\sum_{l\ne j}\kappa_{jl}(q_l-q_j)
\mbox{ for } j=0,1,2,\dots,N-1.
\end{equation}
Since we are considering choreographic motions,
it is sufficient to consider 
the following single equation
\begin{equation}\label{eqOfMotion2}
\ddot{q}_0 
=\sum_{j\ne 0}\kappa_{0j}(q_j-q_0).
\end{equation}

\section{Necessary conditions for $p$ and $N$, and constants of motion}\label{Sec: Nece}
In this section, we derive two necessary conditions on $p$ and $N$ for the $N$-body choreographic motion \eqref{choreograph1}  to satisfy  the equations of motion. Our system is clearly invariant under 
translation and  rotation, which implies that 
any solution must have a constant center of mass and a  constant angular momentum.  To analyze the problem from a broader perspective, we consider a general class of potentials of the form
\[
V=V\left(|q_0-q_1|^2,\dots,|q_i-q_j|^2,\dots\right).
\]
 Recall that we have excluded that $p= 0, \pm 1$.

\subsection{Necessary condition to have constant center of mass}
The center of mass is the first moment of mass normalized by total mass: $\frac{\sum m_i q_i}{\sum m_i}$, and it is constant if and only if the first moment of mass  $g(t)$ is constant, where 
\begin{equation*}
g(t)=\sum_{k}m_k q_k = ae^{it}\sum_{k=0,1,2,\dots,N-1}\omega^k+be^{ipt}\sum_{k=0,1,2,\dots,N-1}\omega^{kp}. 
\end{equation*}


If $p/N \notin \mathbb{Z}$,
then $\omega^p\ne 1$,
and
\begin{equation*}
g(t)
=ae^{it}\,\frac{\omega^N -1}{\omega-1}+be^{ipt}\,\frac{\omega^{pN}-1}{\omega^p-1}
=0.
\end{equation*}
Otherwise, $\omega^p=1$ and
\begin{equation*}
g(t)=bN e^{ipt} \ne \mbox{ constant},
\end{equation*}
 since $p=0$ has been excluded. 
Hence,  when $p\ne 0$,  the  center of mass is constant if and only if 
\[ p/N \notin \mathbb{Z}.\]
This is the \emph{first necessary condition on $p$ and $N$}. If it is satisfied,   
then the center of mass is at the origin.

\subsection{Necessary condition to have constant angular momentum,
and its consequences}
Now let us study  the condition for  constant  angular momentum.
 
Consider
\begin{equation} \notag
\begin{split}
\mathcal{C}
&=\sum_k z_k^\dag \dot{z}_k
=\sum_k (x_k \dot{x}_k+y_k \dot{y}_k)
	+i\sum_k(x_k \dot{y}_k-y_k\dot{x}_k)\\
&=2^{-1}\dot{I}+ic,
\end{split}
\end{equation}
where
\begin{equation} \notag
I=\sum_k (x_k^2+y_k^2),
\end{equation}
is the moment of inertia, and
\begin{equation} \notag
c=\sum_k (x_k \dot{y}_k-y_k\dot{x}_k),
\end{equation}
is the angular momentum.
By a simple calculation, we get
\begin{equation}\notag
\mathcal{C}
=i (a^2+p b^2)N
	+abi\sum_k \left(pe^{i(p-1)t}\omega^{(p-1)k}+e^{-i(p-1)t}\omega^{-(p-1)k}\right).
\end{equation}
Therefore, if $\omega^{p-1}\ne 1$ then $\mathcal{C}=i (a^2+p b^2)N$. 
Otherwise
$$\mathcal{C}=i (a^2+p b^2)N+i abN(p e^{i(p-1)t}+e^{-i(p-1)t}).$$

Note that the angular momentum $c$ is the imaginary part of $\mathcal{C}$.
Therefore, for the case $\omega^{p-1}= 1$, the angular momentum is equal to
\[(a^2+p b^2)N+abN(p+1)\cos\Big((p-1)t\Big)\ne \mbox{ constant,}\]
 since $p=\pm 1$ has been excluded.  
Therefore,  when $p\ne \pm 1$, 
the angular momentum is constant if and only if $\omega^{p-1}\ne 1$,
which is equivalent to 
\[ (p-1)/N \notin \mathbb{Z}.  \]
This is the \emph{second necessary condition on $p$ and $N$}. 

In this case,
we get $c=(a^2+p b^2)N$ and $I=$ constant.
By a simple calculation yields $I=\sum_k z^\dag_k z_k =(a^2+b^2)N$.

 

Similarly, the kinetic energy is also constant,
\begin{equation*}
\begin{split}
2K
&=\sum_k \dot{z}^\dag_k \dot{z}_k
=(a^2+p^2b^2)N+\sum_k \left(e^{i(p-1)t}\omega^{(p-1)k}+e^{-i(p-1)t}\omega^{-(p-1)k}\right)\\
&=(a^2+p^2b^2)N.
\end{split}
\end{equation*}
Since the total energy is constant, the potential energy $V$ must be constant.




\subsection{Necessary conditions to satisfy the equations of motion}
Before closing this section,
we summarize the necessary conditions.

\begin{proposition}\label{necessaryConditionToSatisfyEquationsOfMotion}
Consider the equations of motion that has the translation and rotation invariance.
The necessary conditions for the choreographic motion \eqref{choreograph1} on the $p$-lima\c{c}on
($p\ne 0, \pm 1$)
to satisfy the equations of motion
are $p/N, (p-1)/N \notin \mathbb{Z}$.
\end{proposition}

\begin{proof}
If $p/N, (p-1)/N \in \mathbb{Z}$ (and $p\ne 0, \pm 1$),
the center of mass and the angular momentum of the motion are not constant.
Therefore, this motion cannot satisfy the equations of motion.
\end{proof}

\section{Main Result I} \label{Sec:Main-1}  

In this section, we present our first solution to the main problem stated in the introduction. We begin by rewriting the equation \eqref{eqOfMotion2} and then derive an additional necessary condition on \( p \) and \( N \). Together with the two conditions obtained in Section \ref{Sec: Nece}, we show that these constraints are not only necessary but also sufficient. In other words, if these conditions on \( p \) and \( N \) hold, there exist values of  
\[
\kappa_1, \dots, \kappa_{\lfloor \frac{N}{2} \rfloor}
\]  
such that the \( N \)-body choreographic motion on a \( p \)-limaçon satisfies \eqref{eqOfMotion1}.  

Throughout this section, unless otherwise specified, we set \( n = \lfloor N/2 \rfloor \).

\subsection{Equations of motion}\label{Subsec:equ}
Note that  equation (\ref{eqOfMotion2}) can be written as
\begin{equation}\label{eqOfMotion3}
\begin{split}
\ddot{z}_0 
&= \sum_{j\ne 0} \kappa_{0j}(z_j-z_0)\\
&=\kappa_1\Big((z_1-z_0)+(z_{N-1}-z_0)\Big)
	+\kappa_2(\Big((z_2-z_0)+(z_{N-2}-z_0)\Big)
	+\dots\\
&=\kappa_1(z_1+z_{N-1}-2z_0)
	+\kappa_2((z_2+z_{N-2}-2z_0)
	+\dots
\end{split}
\end{equation}
The last term is given by, 
\begin{equation}\notag
\begin{cases}
\kappa_n (z_n-z_0) & \mbox{ for even } N,\\
\kappa_n (z_n+z_{N-n}-2z_0) & \mbox{ for odd } N.
\end{cases}
\end{equation}
See figure \ref{figKappas}.
Since  $\omega^N=1$,
\begin{equation} \notag
z_{N-j}
=ae^{it}\omega^{N-j}+be^{ipt}\omega^{p(N-j)}
=ae^{it}\omega^{-j}+be^{ipt}\omega^{-pj}.
\end{equation}
Therefore,
the equation of motion (\ref{eqOfMotion2}) is reduced to the condition on the parameters $p, N$, and $\kappa_i$ given by,
\begin{equation}\label{eqOfMotion5}
\begin{split}
-ae^{it}-p^2be^{ipt}
=&ae^{it}\left(
	\kappa_1(\omega+\omega^{-1}-2)+\kappa_2(\omega^2+\omega^{-2}-2)\dots
	\right)\\
	&+
	be^{ipt}\left(
	\kappa_1(\omega^p+\omega^{-p}-2)+\kappa_2(\omega^{2p}+\omega^{-2p}-2)\dots
	\right).
\end{split}
\end{equation}
The last term in the right hand side is
\begin{equation}\notag
\begin{cases}
ae^{it}\kappa_n (\omega^n-1) + be^{ipt}\kappa_n \big(\omega^{pn}-1\big)& \mbox{ for even } N,\\
ae^{it}\kappa_n (\omega^n+\omega^{-n}-2) + be^{ipt}\kappa_n (\omega^{pn}+\omega^{-pn}-2)
				& \mbox{ for odd } N.
\end{cases}
\end{equation}

Since $e^{it}$ and $e^{ipt}$ are linear independent functions in $t$,
the coefficients of $ae^{it}$ and $be^{ipt}$ must be balanced independently.
Therefore, the equation (\ref{eqOfMotion5}) yields
\begin{equation}\label{eqOfMotion6EvenN}
\left(\begin{array}{c}
	-1\\
	-p^2
\end{array}\right)
=\left(\begin{array}{cccc}
	\omega+\omega^{-1}-2& \omega^2+\omega^{-2}-2 & \dots &\omega^n-1 \\
	\omega^p+\omega^{-p}-2& \omega^{2p}+\omega^{-2p}-2 & \dots &\omega^{pn}-1
\end{array}\right)
\left(\begin{array}{c}\kappa_1 \\\kappa_2 \\\vdots \\\kappa_n\end{array}\right)
\end{equation}
for even $N$,
and
\begin{equation}\label{eqOfMotion6OddN}
\left(\begin{array}{c}
	-1\\
	-p^2
\end{array}\right)
=\left(\begin{array}{cccc}
	\omega+\omega^{-1}-2& \omega^2+\omega^{-2}-2 & \dots &\omega^n+\omega^{-n}-2 \\
	\omega^p+\omega^{-p}-2& \omega^{2p}+\omega^{-2p}-2 & \dots &\omega^{pn}+\omega^{-pn}-2 
\end{array}\right)
\left(\begin{array}{c}\kappa_1 \\\kappa_2 \\\vdots \\\kappa_n\end{array}\right)
\end{equation}
for odd $N$.

As a direct corollary, we see that \emph{the values of $a, b$ are irrelevant}. 
Note also that when $N$ is even, we have  $\omega^{n}=-1$.  Hence,  the equations (\ref{eqOfMotion6EvenN}) and (\ref{eqOfMotion6OddN})
are invariant under $p \to -p$.
This invariance yields the following lemma.

\begin{lemma}[Symmetry for $p$]\label{symForP}
If an $N$--body choreographic motion on a $p$-lima\c{c}on curve is the solution
of the equations of motion \eqref{eqOfMotion1} with force coefficients  $\{\kappa_1,\kappa_2,\dots,\kappa_n\}$, then the $N$--body choreographic motion on a $(-p)$-lima\c{c}on curve
solves  the same equations of motion.
\end{lemma}
\begin{proof}
Follows immediately from the above equations.
\end{proof}

%
%
%

\begin{proposition}\label{anotherNecessaryCondition} A necessary condition for the choreographic motion \eqref{choreograph1} on the $p$-lima\c{c}on ($p\ne 0, \pm 1$) to satisfy 
the equations of motion \eqref{eqOfMotion1} is $(p+1)/N \notin \mathbb{Z}$.
\end{proposition}
\begin{proof}
If a motion on $p$-lima\c{c}on with $p=p_0$ satisfies the equations of motion \eqref{eqOfMotion1},
the motion on $p=-p_0$ also satisfy the same equations by the Lemma \ref{symForP}.
Then by the Proposition \ref{necessaryConditionToSatisfyEquationsOfMotion},
$p=-p_0$ must satisfy $(p-1)/N=(-p_0-1)/N \notin \mathbb{Z}$.
Which is equivalent to $(p_0+1)/N \notin \mathbb{Z}$.
\end{proof}

This proposition shows that
although   the choreographic motion \eqref{choreograph1} on the  lima\c{c}ons with $p \equiv -1 \mod N$ satisfy
the necessary conditions for the center of mass and the angular momentum,
they fail to satisfy the equations of motion. 

\subsection{Necessary and sufficient conditions for $p$ and $N$} \label{Subsec:Main1}

Now we are in conditions to state and prove the first  main result of this paper.

\begin{theorem}\label{thm1} Consider  the $N$--body problem  \eqref{eqOfMotion1}. 
The necessary and sufficient condition for $p$ to yield  a choreography on a $p$-lima\c{c}on curve
 ($p\ne 0, \pm 1$)  is that
\begin{equation}\label{condition1}
p/N, (p\pm 1)/N \notin \mathbb{Z}. 
\end{equation}
\end{theorem}

\begin{proof}

\emph{Necessary condition}: The condition $p/N, (p-1)/N \notin \mathbb{Z}$ was already proved 
in Proposition \ref{necessaryConditionToSatisfyEquationsOfMotion}
in Section \ref{Sec: Nece}. The condition $(p+1)/N \notin \mathbb{Z}$ was already proved in Proposition \ref{anotherNecessaryCondition}.

\emph{Sufficient condition}: When $p/N, (p\pm 1)/N \notin \mathbb{Z}$, we show that we can always find $\kappa_1, \dots, \kappa_n$ such that the $N$--body choreographic motion on a $p$-lima\c{c}on curve solves the equations of motion. 

We start by splitting the coefficient matrix in (\ref{eqOfMotion6EvenN}) or (\ref{eqOfMotion6OddN})
into the coefficient matrix $M_t$ for $\kappa_1, \kappa_2$,
and $\widetilde{M}$ for the other $\kappa$'s.
Namely,
\begin{equation}
\left(\begin{array}{c}
	-1\\
	-p^2
\end{array}\right)
=M_t
\left(\begin{array}{c}\kappa_1 \\\kappa_2\end{array}\right)
+\widetilde{M}\left(\begin{array}{c}\kappa_3 \\\kappa_4\\\vdots\\\kappa_{\lfloor N/2 \rfloor}\end{array}\right),
\end{equation}
where
\begin{equation}\notag
M_t=\left(\begin{array}{cc}
	\omega+\omega^{-1}-2& \omega^2+\omega^{-2}-2 \\
	\omega^p+\omega^{-p}-2& \omega^{2p}+\omega^{-2p}-2
\end{array}\right).
\end{equation}
and $\widetilde{M}$ is the $2 \times (\lfloor N/2 \rfloor-2)$ matrix,
where each of its entries is real and depends on $p,N$.

The determinant of $M_t$ is 
\begin{equation}\notag
\begin{split}
&(\omega -1)^2 \omega ^{-2 (p+1)} \left(\omega ^p-1\right)^2 \left(\omega ^{2 p+1}-\omega ^{p+2}-\omega ^p+\omega \right) \\
=&(\omega -1)^2 \omega ^{-2 (p+1)} \left(\omega ^p-1\right)^2  \omega ^{p+1}  \left(\omega ^{p}+\omega ^{-p}- \omega -\omega ^{-1} \right) \\
=&2\omega^{-(p+1)}(\omega-1)^2 (\omega^p-1)^2
	\left(\cos\left(\frac{2p\pi}{N}\right)-\cos\left(\frac{2\pi}{N}\right)\right) \neq 0, 
\end{split}
\end{equation}
 if $p/N,(p \pm 1)/N \notin \mathbb{Z}$.

 Therefore, $M_t$ has an inverse, we obtain
\begin{equation}\label{equ:solution_kij}
\left(\begin{array}{c}\kappa_1 \\\kappa_2\end{array}\right)
=M_t^{-1}\left(
		\left(\begin{array}{c}
			-1\\
			-p^2
		\end{array}\right)
		-
		\widetilde{M}\left(\begin{array}{c}\kappa_3 \\\kappa_4\\\vdots\\\kappa_{\lfloor N/2 \rfloor}\end{array}\right)
	\right).
\end{equation}
Thus, $\kappa_1$ and $\kappa_2$ are determined uniquely
by $p,N,\kappa_3,\kappa_4,\dots,\kappa_{\lfloor N/2 \rfloor}$.
This ends the proof of Theorem \ref{thm1}.
\end{proof}

Thus, we have answered the problem proposed in the introduction.  That is, to construct an  $N$-body choreography on a $p$-Lima\c{c}on curve, the values of $a, b$ are irrelevant, the values of $p, N$ have to satisfy the condition \eqref{condition1}, and when the condition is satisfied, the force coefficients 
	$\kappa_1, \ \dots, \ \kappa_n$ are determined by using  \eqref{equ:solution_kij}.

As one example, let  $N=6$, where the potential depends on the three force coefficients  $\kappa_1, \kappa_2, \kappa_3$.  When 
 $p=6k\pm 2$, 
 \eqref{equ:solution_kij} says
\begin{equation}\label{N=6_2}
\kappa_1=\frac{p^2-1}{2}+\kappa_3, \quad
\kappa_2=\frac{-p^2+3}{6}-\kappa_3.
\end{equation}
When $p=6k+3$,  \eqref{equ:solution_kij} says
\begin{equation}\label{N=6_3}
\kappa_1=\frac{p^2}{4}-\frac{\kappa_3}{2},\quad
\kappa_2=\frac{-p^2+4}{12}-\frac{\kappa_3}{2}.
\end{equation}

\begin{remark}
	We observe that \eqref{equ:solution_kij}  determines an affine transformation from $(\kappa_3, \kappa_4, \dots, \kappa_{\lfloor N/2 \rfloor})$ to  $(\kappa_1, \kappa_2)$. So those $(\kappa_3, \kappa_4, \dots, \kappa_{\lfloor N/2 \rfloor})$ that makes $\kappa_1$ or $\kappa_2$  zero form one $(\lfloor N/2 \rfloor-3)$-dimensional hyper-plane. Therefore, in most cases,  all $\kappa_i$'s are non-zero, i.e.,  there is interaction between every pair of particles.  
\end{remark}

%


Theorem \ref{thm1} tells us that
$p \equiv 2,3,\dots, N-2 \mod N$
for given $N$.
Then, what are the admissible  $N$ for given $p\ge 2$ ?
The answer is given in  the following corollary.
Since we have the symmetry for $p \leftrightarrow -p$,
it is enough to consider $p\ge 2$.
\begin{cor}
Let $\mathbb{D}(p)$ be the set of divisors of $p$ and  $p\pm 1$. 
In the $N$--body problem  \eqref{eqOfMotion1}, for a fixed $p$, 
the necessary and sufficient condition for $N$ to ensure a choreography  on a $p$-lima\c{c}on curve  is 
 $N \notin \mathbb{D}(p)$.
\end{cor}

\begin{proof}
Obvious from condition \eqref{condition1}.
\end{proof}

Note that  $\{1,2,3,  p-1,p,p+1\} \subseteq \mathbb{D}(p)$ for any $p\ge 2$,
so $N \geq 4$. Note also that 
any  $N\ge p+2$ bodies can be put on the $p$-lima\c{c}on curve.
The appropriate $N$ are, for example,
$\{\mbox{all }N\ge p+2\}$ for $2\le p\le 5$,
$\{4, \mbox{all }N\ge 8\}$ for $p=6$, 
$\{5, \mbox{all }N\ge 9\}$ for $p=7$, 
$\{5, 6, \mbox{all }N\ge 10\}$ for $p=8$, etc.

		
\subsection{Remarks on the potential and Saari's conjecture}\label{remarksVandSaari}
It has been shown that the $N$-body choreography on $p$-lima\c{c}on  with $p/N, (p\pm 1)/N \notin \mathbb{Z}$
satisfies the equations of motion (\ref{eqOfMotion1}),
and that the center of mass is set to the origin, the angular momentum, the moment of inertia,
and the kinetic energy are constant.

Since the total energy is constant, the potential energy must be constant.
The first remark is on its value.
\begin{remark}
For the lima\c{c}on solution under $V$ in (\ref{Lagrangian}),
the value of the potential is given by
$V=K=2^{-1}(a^2+p^2b^2)N$.
\end{remark}
The reason is the following:
The constant moment of inertia and the well-known Lagrange-Jacobi identity yields
\begin{equation*}
0
=\ddot{I}
=2\sum_k |\dot{q}_k|^2+2\sum_k q_k \cdot\ddot{q}_k
=4K-2\sum_k q_k\cdot \frac{\partial V}{\partial q_k}
=4(K-V).
\end{equation*}
Namely, $V=K$.

The second remark is for Saari's conjecture.
\begin{remark}
	In the Newtonian $N$-body problem, Saari conjectured that relative equilibria are the only solutions where the moment of inertia with respect to the center of mass remains constant \cite{Saari1970}.  
	For the $N$-body problem \eqref{eqOfMotion1}, we 
	have established
	the existence of choreographies on $p$-lima\c{c}ons curves ($p\ne 0, \pm 1$). For each of such motion, the moment of inertia remains constant, providing counterexamples to Saari’s conjecture in the context of  the $N$-body problem \eqref{eqOfMotion1}.
\end{remark}

\section{Main Result II} \label{Sec:Main-2}

In this section, we present our second result addressing the main problem described in the introduction.  We impose the following condition on the force coefficients:
\begin{equation}\label{condition2}
\kappa_o = \kappa_1 = \kappa_3 = \kappa_5 = \dots, \quad \kappa_e = \kappa_2 = \kappa_4 = \kappa_6 = \dots.
\end{equation}
This restriction has meaningful interpretations, which we will explain in the first subsection. We then proceed to prove our second main result.

\begin{theorem}\label{thm2}
	Consider the \( N \)-body problem \eqref{eqOfMotion1} under the restriction \eqref{condition2}. The necessary and sufficient condition for \( p \) to obtain a choreography on a \( p \)-limaçon curve  ($p\ne 0, \pm 1$) is as follows:
	\begin{itemize}
		\item For even \( N \), the condition is \( p/N, (p \pm 1)/N \notin \mathbb{Z} \) and \( p \equiv N/2 \mod N \). Then, \( \kappa_o = p^2/N \) and \( \kappa_e = (2 - p^2)/N \).
		\item For odd \( N \), the condition is \( p/N, (p \pm 1)/N \notin \mathbb{Z} \). In this case, \( \kappa_o \) and \( \kappa_e \) are uniquely determined by \( p \) and \( N \).
	\end{itemize}
\end{theorem}

To illustrate the above result, see \eqref{N=6_2} and \eqref{N=6_3} for \( N = 6 \).

In this section, unless otherwise specified, \( n \) is taken to be \( \lfloor N/2 \rfloor \). According to Theorem \ref{thm1}, we assume that \( p/N, (p \pm 1)/N \notin \mathbb{Z} \) throughout this section.

\subsection{Motivation}

Consider an $N$-body system under the influence of a gravitational type force and   an electrostatic type force. Let the position of the $j$-th body be  $q_j$, $j=0,1,2,\dots,N-1$. Assume all masses are equal, and 
the charges are  $e_j=(-1)^j e$ for $j=0,1,2,\dots,N-1$. 
Namely, $m_0=m_1=\dots=m_{N-1}=m$, $e_0=e_2=\dots=e$ and $e_1=e_3=\dots=-e$. More precisely, 
the Lagrangian of the system is  given by
\begin{equation}\label{Lagrangian2}
\begin{split}
L&=\sum_j \frac{m}{2}|\dot{q}_j|^2
-\frac{1}{2}\sum_{j<l} m^2 |q_j-q_l|^{2}
+\frac{1}{2}\sum_{j<l} e_j e_l |q_j-q_l|^{2}\\
&=\sum_j \frac{m}{2}|\dot{q}_j|^2
-\frac{1}{2}\sum_{j<l} \kappa_{jl} |q_j-q_l|^{2}, 
\end{split}
\end{equation}
where 
\begin{equation}\label{equ:kappa}
\kappa_{jl}= m^2 - (-1)^{j+l}e^2
\end{equation}

When  $N$ is odd, \eqref{equ:kappa} implies that $\kappa_{0,1}\ne \kappa_{N-1,0}$,  which breaks the $Z_N$ invariance.
However, when $N$ is even, set 
\[\kappa_o = m^2+e^2,\quad \kappa_e=m^2-e^2.
\]
Then \eqref{equ:kappa} yields the following, 
\begin{align*}
&\kappa_o=\kappa_1=\kappa_{0,1}=\dots=\kappa_{j,j+1}=\dots=\kappa_{N-1,0},\\
&\kappa_e=\kappa_2=\kappa_{0,2}=\dots=\kappa_{j,j+2}=\dots=\kappa_{N-2,0}=\kappa_{N-1,1}, \\
&\kappa_o=\kappa_3=\kappa_{0,3}=\dots=\kappa_{j,j+3}=\dots=\kappa_{N-1,2},\\
&\dots
\end{align*}

Therefore, when $N$ is even, for  the system with the  Lagrangian \eqref{Lagrangian2}, to construct  an $N$-body  choreography on the $p$-lima\c{c}on curve, $q_k(t)=q(t+k \frac{2\pi}{N}), k=0,1,2,\dots, N-1$, 
is equivalent to solve the equation \eqref{eqOfMotion5}, under the restriction \eqref{condition2}.

The above setting provides a meaningful motivation for imposing restriction \eqref{condition2} when 
$N$ is even. 
For odd $N$,  we have not found a direct physical interpretation, but the restriction remains mathematically interesting.
 Note that the restriction itself keeps the $Z_N$ symmetry.  
Therefore, we analyze equation \eqref{eqOfMotion5}  under restriction \eqref{condition2} for all $N\ge 4$
in the following.

Now, we rewrite the equation \eqref{eqOfMotion5}, or equivalently, the system \eqref{eqOfMotion6EvenN} and \eqref{eqOfMotion6OddN}, under the restriction \eqref{condition2}. 
Recall that $n=\lfloor N/2\rfloor$. 
The restriction \eqref{condition2} is 
\[
\left(\begin{array}{c}\kappa_1 \\\kappa_2 \\\vdots \\\kappa_n\end{array}\right)
=\left(\begin{array}{cc}1 & 0 \\0 & 1 \\1 & 0 \\0 & 1 \\\vdots & \vdots\end{array}\right)
\left(\begin{array}{c}\kappa_o \\\kappa_e\end{array}\right).
\]
Let 
\[
\begin{split}
M
&=\left(\begin{array}{cccc}
\omega+\omega^{-1}-2& \omega^2+\omega^{-2}-2 & \dots &\dots \\
\omega^p+\omega^{-p}-2& \omega^{2p}+\omega^{-2p}-2 & \dots &\dots
\end{array}\right)
\left(\begin{array}{cc}1 & 0 \\0 & 1 \\1 & 0 \\0 & 1 \\\vdots & \vdots\end{array}\right)\\
&=\left(\begin{array}{cc}M_{11} & M_{12} \\M_{21} & M_{22}\end{array}\right).
\end{split}
\]
The  equation  \eqref{eqOfMotion5}  is  reduced into
\begin{equation}\label{eqOfMotion7}
\left(\begin{array}{c}
-1\\
-p^2
\end{array}\right)
=M\left(\begin{array}{c}\kappa_o \\\kappa_e\end{array}\right).
\end{equation}

The system \eqref{eqOfMotion7} can be solved by finding the four entries of $M$.  Depending on whether $n = \lfloor N/2 \rfloor$ is odd or even, the last column of the coefficient matrix in \eqref{eqOfMotion6EvenN} or \eqref{eqOfMotion6OddN} corresponds to either the first or the second column of $M$. Consequently, four distinct cases arise depending on the value of \( N \bmod 4 \).
Besides, for even $N$,  $M_{21}$ and $M_{22}$ have the term $\sum_k \omega^{2pk}$,
therefore we have to consider $\omega^{2p}=1$ or not.
However, the following observation will greatly simplify the situation. 

\begin{lemma}\label{lem:first-e-value}
If $p/N\notin \mathbb{Z}$, then 	$M\left(\begin{array}{c}1 \\1\end{array}\right)=-N \left(\begin{array}{c}1 \\1\end{array}\right)$. 
\end{lemma}
\begin{proof}
It is enough to 
verify that $M_{11}+M_{12}=M_{21}+M_{22}=-N. $ 
It is true since 
\begin{align*}
&M_{11}+M_{12}= \sum_{k=0,1,2,\dots,N-1}\left(\omega^{k}-1\right) =-N +\sum_{k=0,1,2,\dots,N-1} \omega^{k}=-N,\\
&M_{21}+M_{22}= \sum_{k=0,1,2,\dots,N-1}\left(\omega^{kp}-1\right) =-N +\sum_{k=0,1,2,\dots,N-1} \omega^{kp}=-N. 
\end{align*}
In the last equation, we used the fact that $\omega^p\ne 1$. This completes the proof. 
\end{proof}

\subsection{Proof of Theorem \ref{thm2}}

\begin{proof}
	Note that to construct  an $N$-body  choreography  on the $p$-lima\c{c}on curve, $q_k(t)=q(t+k \frac{2\pi}{N}), k=0,1,2,\dots, N-1$, under the restriction \eqref{condition2}, 
	is equivalent to solve the equation \eqref{eqOfMotion7}.

	1. $N$ is even. Then $N=2n$. Since $\omega^2=e^{i 4\pi/2n}= e^{i 2\pi/n}\ne 1$ for $n\ge 2$, we obtain 
	\[  M_{11}=\sum_{k=0,1,\dots,n-1}\left(\omega^{2k+1}-1\right)=-n+\frac{\omega(\omega^{2n}-1)}{\omega^2-1}=-n=-N/2.  \]
	The value of $M_{22}=\sum_{k=0,1,\dots,n-1}\left(\omega^{2kp}-1\right)$ depends on whether $\omega^{2p}$ equals $1$ or not. 
	
	\emph{Case I}:  $\omega^{2p}=1$. Then $p\equiv N/2 (\mod N)$  since $p/N \notin \mathbb{Z}$, and 
\[
M_{22}=\sum_{k=0,1,\dots,n-1}\left(\omega^{2kp}-1\right)=0. 
\]
By Lemma \ref{lem:first-e-value}, we obtain 
\[
M=\left(\begin{array}{cc}-N/2 & -N/2 \\-N & 0\end{array}\right).
\]
It is easy to see that solution of the system \eqref{eqOfMotion7}
is 
\[    \kappa_o=p^2/N, \quad \kappa_e=(2-p^2)/N.   \]

\emph{Case II}:  $\omega^{2p}\ne 1$.  Then,
\[
M_{22}=-n+\frac{\omega^{2np}-1}{\omega^{2p}-1}=-n=-N/2.
\]
By Lemma \ref{lem:first-e-value}, we obtain 
\[
M=\left(\begin{array}{cc}-N/2 & -N/2 \\-N/2 &-N/2\end{array}\right).
\]
Then the system \eqref{eqOfMotion7} is inconsistent since $p^2\ne 1$. 

2. $N$ is odd. Then $N=4m+1$ or $N=4m+3$. 
We show that  the rank of $M$ is $2$ by  showing that both   eigenvalues of $M$ are nonzero.  

By Lemma \ref{lem:first-e-value}, $M$ has one eigenvalue being $-N$.  Let the other eigenvalue be $\lambda$,
then 
\[
\lambda=N+M_{11}+M_{22}=M_{11}-M_{21}. 
\]

\emph{Case I}: $N=4m+1$. Note that  $n=(N-1)/2 = 2m$ even, then
\begin{align*}
&M_{11}=\sum_{k=1}^{m}\left(\omega^{2k-1}+\omega^{-(2k-1)}-2\right)&= -2m + \frac{\omega(\omega^{2m}-1)}{\omega^2-1}+\frac{\omega^{-1}(\omega^{-2m}-1)}{\omega^{-2}-1}.\end{align*}
Using $\omega^{-2m}=\omega^{2m+1}$, we obtain 
\[ M_{11}=-2m + \frac{\omega^{2m+1}-\omega^{1-2m}}{\omega^2-1} 
 =-2m - \frac{\omega^{2m+1}}{\omega+1}
=-2m - \frac{\omega^{N/2}}{\omega^{1/2}+\omega^{-1/2}}. \]
 Note that \( M_{21} \) is obtained from \( M_{11} \) by replacing \( \omega \) with \( \omega^p \). Then
\[
\begin{split}
\lambda=-\frac{-1}{2\cos(\pi/N)}+\frac{(-1)^p}{2\cos(\pi p/N)}.
\end{split}
\]
Since $(p\pm 1)/N \notin \mathbb{Z}$, we conclude that 
 $\lambda\ne 0$. 

\emph{Case II}: $N=4m+3$. Note that  $n=(N-1)/2 = 2m+1$ odd. Then 
\[   M_{11}=\sum_{k=1}^{m+1}\left(\omega^{2k-1}+\omega^{-(2k-1)}-2\right)= -2(m+1) + \frac{\omega(\omega^{2m+2}-1)}{\omega^2-1}+\frac{\omega^{-1}(\omega^{-2m-2}-1)}{\omega^{-2}-1}. \]
Using $\omega^{-2m-1}=\omega^{2m+2}$,  we obtain 
\[
M_{11} = -2(m+1) + \frac{\omega^{2m+3}-\omega^{-1-2m}}{\omega^2-1} =-2(m+1)  + \frac{\omega^{2m+2}}{\omega+1}
=-2(m+1) +  \frac{\omega^{N/2}}{\omega^{1/2}+\omega^{-1/2}}. \]
Similar to the above case, we obtain 
\[
\begin{split}
\lambda=\frac{-1}{2\cos(\pi/N)}-\frac{(-1)^p}{2\cos(\pi p/N)}\ne 0. 
\end{split}
\]

Therefore, when $N$  is odd,   the rank of $M$ is $2$.  Thus, there is always a unique solution to the system \eqref{eqOfMotion7}.   This ends the proof of Theorem \ref{thm2}.
\end{proof}

\section{Collisions} \label{Sec:Collision}

Given a suitable $p$ and $N$, the parameters \( a \) and \( b \) are free, provided that collisions are not a concern. However, under certain combinations of \( p \), \( N \), \( a \), and \( b \), collisions may occur, as demonstrated in the following example.

Consider the special lima\c{c}on curve  with $p=2$ and $a=b$, 
\begin{equation}\notag
z(t)=a\left(e^{i t}+e^{2 i t}\right).
\end{equation}
This curve has a self-intersecting point at $(-a, 0)$,  (see  Figure \ref{fig:collision}), which corresponds to
$t=\pm 2 \pi / 3$. Therefore, two bodies $z_k$ and $z_{k^{\prime}}$ will collide at this point if they
satisfy $\left(k-k^{\prime}\right) 2 \pi / N=2 \pi / 3$, namely if $k-k^{\prime}=N / 3$. Thus, for this special 2-lima\c{c}on curve, the choreographic motions with $N \in 3 \mathbb{Z}$ will have collision.
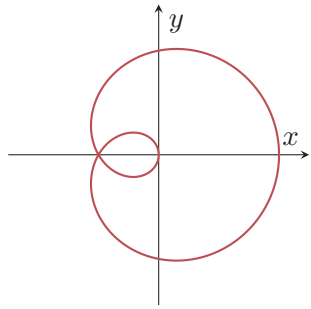
\begin{figure}[h]
	\begin{center}
		\begin{tikzpicture}
			\begin{axis}[width=4cm, height=4cm,
		scale only axis,
			xmin=-2.5, xmax=2.5,
			ymin=-2.5, ymax=2.5,
			xlabel={$x$}, ylabel={$y$},
			samples=200,
			domain=0:2*pi,
			xtick=\empty, ytick=\empty,
			enlargelimits=false,
			axis lines=middle
			]
			\addplot[purple, thick, samples=300, domain=0:2*pi, smooth] (
			{cos(deg(x)) + cos(2*deg(x))},
			{sin(deg(x)) + sin(2*deg(x))}
			);
		\end{axis}
	\end{tikzpicture}
	\end{center}

	\caption{The 2-lima\c{c}on curve for $a=b$. }
	\label{fig:collision}
 \end{figure}

Obviously, if a $p$-lima\c{c}on curve is simple, then any choreographic motion on it has  no collision.   Collisions occur only at the self-intersecting points and self-touching points, i.e.,  a point on the curve that corresponds to multiple instants    $0 \leq t_1<t_2<\cdots<t_l< 2\pi$. Call them \emph{multiple points}. 
The number and location of multiple points of a $p$-lima\c{c}on curve depend on the values of $p, a, b$. 

\begin{lemma}\label{lem:collision}
	Consider a  $p$-lima\c{c}on curve $z(t)=a e^{it} + b  e^{ip t}$. Let $P$ be a multiple point and let  the corresponding instants be $\{ t_1, \cdots, t_\ell\}$, where $0 \leq t_1<t_2<\cdots<t_\ell< 2\pi$. 
	Then, an $N$-body choreographic motion on this $p$-lima\c{c}on curve will have  collision at $P$, if and only if there exist one pair of instants $t_j, t_l (t_j<t_l)$   such that the time  difference 
	$t_l-t_j$ belongs to the set 
	$$\{ \frac{2\pi}{N}, \frac{4\pi}{N}, \cdots, \frac{2(N-1)\pi}{N} \}.$$ 
\end{lemma}

\begin{proof}
	A $N$--body choreographic motion on this $p$-lima\c{c}on curve has a collision at $P$ if and only if there exists a pair of  instants $ t_j, t_l$
	$( t_j < t_l$), some indices $ 1 \leq k < k' \leq N $, and a time $ t^* \in (0, 2\pi)$ such that $ z_k(t^*) = z_{k'}(t^*) = P$, or in other words if
	\begin{equation}\label{equ:collision}
	t^* + k  \frac{2\pi}{N}= t_j, \ \ t^* + k' \frac{2\pi}{N} = t_l.  
	\end{equation} 
	Therefore, it is necessary that $t_l-t_j \in  \{ \frac{2\pi}{N}, \frac{4\pi}{N}, \cdots, \frac{2(N-1)\pi}{N} \}$. 
	
	On the other hand, suppose that there exists a pair of  instants 
	$t_j, t_l$ with $t_j<t_l$ such that $t_l-t_j =  s\frac{2\pi}{N} $, where $1\le s\le N-1$. It is clear that there are $k\in \{0, \cdots, N-1\}$ such that 
	\[  t_j - k\frac{2\pi}{N} \in (0, 2\pi).\]
	Let be $t^*= t_j - k\frac{2\pi}{N} $. Then 
	\[ t^*= t_l - k\frac{2\pi}{N}  + t_j -t_l=t_l - k' \frac{2\pi}{N} ,     \]  
	where $k' =k+s$. Hence, \eqref{equ:collision} is satisfied, i.e., at $t=t^*$, $q_k$ will collide with $q_{k'}$ at the point $P$. 
\end{proof}

For  a given $p$--lima\c{c}on curve determined by the values of $p, a$, and $b$,  although collisions  may occur in $N$-body choreographic motions,   it turns out that such $p$-lima\c{c}ons are quite rare. 

\begin{proposition}\label{Prop1}
	For any given $p, N$, there are at most $2(N-1)$ values of $a/b$ such that an
	$N$--body choreographic motion  on the $p$-lima\c{c}on curve $z(t)=a e^{it} + b  e^{ip t}$  experiences collisions. 
\end{proposition}

\begin{proof}
	Suppose that $P$ is a multiple point of $z(t)=a e^{it} + b  e^{ip t}$ and two  corresponding instants are $0 \leq t_1<t_2<2\pi$. Then $z(t_1)=z(t_2)$, and it follows that 
	\begin{equation}\notag
	1- e^{i (t_2-t_1)}=\frac{b}{a} e^{i (p-1)t_1} (e^{i p (t_2-t_1)}-1).   
	\end{equation}
	Since $1- e^{i (t_2-t_1)}\ne 0$, then 
	\begin{equation}\notag
	\begin{split}
	\frac{a}{b}
	&=-e^{i(p-1)t_1}\frac{e^{ip(t_2-t_1)}-1}{e^{i(t_2-t_1)}-1}
	=-e^{i(p-1)(t_1+t_2)/2}
	\,\,\frac{e^{ip(t_2-t_1)/2}-e^{-ip(t_2-t_1)/2}}{e^{i(t_2-t_1)/2}-e^{-i(t_2-t_1)/2}}\\
	&=-e^{i(p-1)(t_1+t_2)/2}\,\,\frac{\sin(\pi pk/N)}{\sin(\pi k/N)}.
	\end{split}
	\end{equation}
	Since $a/b$ and $\sin(\pi pk/N)/\sin(\pi k/N)$ are real,
	the factor $-e^{i(p-1)(t_1+t_2)/2}$ must be $\pm 1$.
	Therefore,
	\begin{equation}\label{theRatio}
	\frac{a}{b}
	=\pm \frac{\sin(\pi pk/N)}{\sin(\pi k/N)},
	\,\,k=1,2,3,\dots,N-1.
	\end{equation}

	Therefore, 
	for given $p$ and $N$,
	the number of the quotient $a/b$ 
	that gives collision
	is at most $2(N-1)$.
\end{proof}

\section{Additional constants}\label{Sec:Constants} 

For any  motion that satisfies the equations  \eqref{eqOfMotion1}, the first moment  of mass,  the angular momentum, and the total energy are constants of motion.  In the case of  $N$-body choreographies on a $p$-lima\c{c}on curve, the $Z_N$ symmetry yields additional constants, which we explore in this section.  According to Theorem \ref{thm1}, we assume that \( p/N, (p \pm 1)/N \notin \mathbb{Z} \).

\subsection{Parts of the potential}
Recall that 
\[
z_k^\dag z_{k+\ell}
=(x_k-iy_k)(x_{k+\ell}+iy_{k+\ell})
=q_k\cdot q_{k+\ell}+i\, q_k\times q_{k+\ell}.
\]
On the other hand,
\[
\begin{split}
z_k^\dag z_{k+\ell}
&=\left(ae^{-it}\omega^{-k}+be^{-ipt}\omega^{-pk}\right)
\left(ae^{it}\omega^{k+\ell}+be^{ipt}\omega^{-p(k+\ell)}\right)\\
&=a^2 \omega^\ell + b^2 \omega^{p\ell}
+ab\left(
e^{i(p-1)t}\omega^{(p-1)k + p\ell}
+e^{-i(p-1)t}\omega^{-(p-1)k + \ell}
\right).
\end{split}
\]
Since  $\omega^{p-1}\ne 1$, then 
\[
\sum_{k=0,1,2,\dots,N-1}z_k^\dag z_{k+\ell}
=N\left(a^2 \omega^\ell + b^2 \omega^{p \ell}\right)
=\mbox{ constant}.
\]
Namely, both
\[\begin{split}
\sum_k q_k\cdot q_{k+\ell}=N\left(a^2 \cos(2\pi \ell/N)+b^2\cos(2\pi p\ell/N)\right),\\
\sum_k q_k\times q_{k+m}=N\left(a^2 \sin(2\pi \ell/N)+b^2\sin(2\pi p\ell/N)\right)
\end{split}
\]
are constant.
Since $I=\sum_k |q_k|^2 =(a^2+b^2)N$, 
\begin{align*}
v^{\pm}_\ell
&=\sum_k |q_k\pm q_{k+\ell}|^2=2I \pm  2\sum_k q_k\cdot q_{k+\ell}\\
&=2N\left(a^2\big(1\pm \cos(2\pi \ell/N)\big)+b^2\big(1\pm \cos(2\pi p\ell/N)\big)\right)
\end{align*}
is constant.

The $v^{-}_\ell$ are the parts of the potential.
The potential energy is given by $V=\frac{1}{2}\sum_{\ell=1,2,\dots,\lfloor N/2 \rfloor}\kappa_\ell V_\ell$,
where $V_\ell=v_\ell^{-}$ for $1\le \ell \le \lfloor N/2 \rfloor-1$, and
\[
V_{\lfloor N/2 \rfloor}=
\begin{cases}
\phantom{2^{-1}}v^{-}_{\lfloor N/2 \rfloor}&\mbox{for odd }N,\\
2^{-1}v^{-}_{N/2}&\mbox{for even }N.
\end{cases}
\]

\subsection{Partial sums of  the  first moment  of mass,  momentum of inertia, angular momentum, and kinetic energy   for non-prime $N$}
If $N$ is not a prime number, say, 
$N=mn$, then 
$Z_N=\{1,\omega,\omega^2,\dots,\omega^{N-1}\}$, where $\omega=e^{i 2\pi/N}$, 
has the subgroup 
$Z_m=\{1,\omega^n, \omega^{2n},\dots,\omega^{(m-1)n}\}$.

Let $s$ be one of the  constants,  $g$ the first moment  of mass, $I$ the momentum of inertia, $c$ the angular momentum, and $K$ the kinetic energy, and $s_m(\ell)$ be the partial sum defined by
\[\begin{split}
s=\sum_{\ell=0,1,2,\dots,n-1}s_m(\ell),
\mbox{ and }
s_m(\ell)=\sum_{k=0,1,2,\dots,m-1} Q_{\ell+kn}.
\end{split}\]
For instance, when $s=g$,  then $Q_{\ell+kn}=q_{\ell+kn}$,  and 


\[
\begin{split}
g_m(\ell)&=a e^{it}\sum_{0\le k\le m-1} \omega^{\ell+kn} + be^{ipt}\sum_{0\le k\le m-1} \omega^{\ell p+knp}\\
&=be^{ipt} \omega^{\ell p}\sum_{0\le k\le m-1} \omega^{knp}. 
\end{split}
\]
Thus, we obtain the following new constants, 
\[
\begin{array}{cc||c|c|c}
s&Q_{\ell+kn} & s_m(\ell) & \omega^{pn}\ne 1 & \omega^{pn}= 1\\
\hline
g& q_{\ell+kn}&g_m(\ell)& 0& bm\, e^{ipt}\omega^{\ell p} \mbox{ and }|g_m(\ell)|^2=(bm)^2\\
\end{array}
\]


Similarly, we compute the partial sums of  the momentum of inertia, the angular momentum, and the kinetic energy.  
We summarize the results without detail calculations.
\[
\begin{array}{cc||l|l}
s& Q_{\ell+kn} &s_m(\ell) &\phantom{2^{-1}}\omega^{(p-1)n}\ne 1\\
\hline
I& |q_{\ell+kn}|^2 &I_m(\ell)&\phantom{2^{-1}}m(a^2+b^2)\\
c& q_{\ell+kn}\times \dot{q}_{\ell+kn} &c_m(\ell)&\phantom{2^{-1}}m(a^2+pb^2)\\
K& 2^{-1}|\dot{q}_{\ell+kn}|^2&K_m(\ell)&2^{-1}m(a^2+p^2b^2)
\end{array}
\]

The following  is another type of constant. When $\omega^{(p-1)n}\ne 1$, 
\[\begin{split}
\tilde{I}_m(\ell)
&=\sum_{0\le i<j \le m-1}|q_{\ell+in}-q_{\ell+jn}|^2
=m I_m(\ell)-|g_m(\ell)|^2\\
&=\begin{cases}
m^2(a^2+b^2)&\mbox{for } \omega^{np}\ne 1, \\
m^2 a^2&\mbox{for } \omega^{np}= 1. 
\end{cases}
\end{split}\]
Note that generally,  $\sum_{\ell=0,1,2,\dots,n-1}\tilde{I}_m(\ell)\ne \sum_{0\le i<j\le N-1}|q_i-q_j|^2$.

\end{document}